\documentclass[reqno,10pt]{amsart}
\numberwithin{equation}{section}
\usepackage{latexsym}
\usepackage{amsmath}
\usepackage{amssymb}
\usepackage{mathrsfs}
\usepackage{graphicx,colordvi}
\usepackage{upgreek}
\usepackage{ifthen}
\usepackage{color}
\usepackage{xcolor}

\usepackage[T1]{fontenc}
\usepackage[latin1]{inputenc}

\numberwithin{equation}{section}

\newtheorem{defi}{Definition}[section]
\newtheorem{theorem}[defi]{Theorem}
\newtheorem{lemma}[defi]{Lemma}

\newtheorem{proposition}[defi]{Proposition}
\newtheorem{remark}[defi]{Remark}

\newcommand{\cS}{\mathcal{S}}
\newcommand{\cA}{{\mathcal A}}

\newcommand{\cE}{{\mathcal E}}
\newcommand{\cH}{{\mathcal H}}

\newcommand{\cL}{{\mathcal L}}

\newcommand{\cP}{{\mathcal P}}

\newcommand{\R}{{\mathbb R}}

\newcommand{\bJ}{\mathbb{J}}
\newcommand{\bK}{\mathbb{K}}
\newcommand{\bN}{\mathbb{N}}

\newcommand{\bR}{\mathbb{R}}
\newcommand{\bS}{\mathbb{S}}

\newcommand{\gf}{g_{\flat}}
\newcommand{\gs}{g^{\sharp}}

\newcommand{\sD}{\mathsf{D}}

\newcommand{\sM}{\mathsf{M}}
\newcommand{\sN}{\mathsf{N}}

\newcommand{\sT}{\mathsf{T}}

\renewcommand{\epsilon}{\varepsilon}

\newcommand{\bb}{{\bb}}
\newcommand{\eps}{\varepsilon}
\newcommand{\la}{\langle}
\newcommand{\ra}{\rangle}

\newcommand{\PH}{\mathbb{P}_H}

\newcommand{\gd}{\mathsf{grad}\,}
\newcommand{\dv}{\mathsf{div}\,}
\newcommand{\rot}{\mathsf{rot}\,}

\newcommand{\Ric}{\mathsf{Ric\,}}
\newcommand{\Aw}{A^\mathsf{w}}

\newcommand{\sw}{\mathsf{w}}

\begin{document}

\title{Coriolis-driven fluid motion on \\ spherical surfaces}

\author{Yuanzhen Shao}
\address{The University of Alabama\\
	Tuscaloosa, Alabama \\
	USA}
\email{yshao8@ua.edu}

\author{Gieri Simonett}
\address{Department of Mathematics\\
        Vanderbilt University\\
        Nashville, Tennessee\\
        USA}
\email{gieri.simonett@vanderbilt.edu}

\author{Mathias Wilke}
\address{Martin-Luther-Universit\"at Halle-Witten\-berg\\
         Institut f\"ur Mathematik \\
         Halle (Saale), Germany}
\email{mathias.wilke@mathematik.uni-halle.de}

\thanks{This work was supported by a grant from the Simons Foundation (\#853237, Gieri Simonett)  and 
a grant from the National Science Foundation (DMS-2306991, Yuanzhen Shao).}



\begin{abstract}
We consider the motion of an incompressible viscous fluid on a sphere, incorporating the effects of the Coriolis force.
We demonstrate that global solutions exist for any divergence-free initial condition with finite kinetic energy.
Furthermore, we show that each solution converges at an exponential rate to a state that is aligned with the rotation of the sphere.
\end{abstract}

\maketitle
\section{Introduction}

\bigskip
In previous work, we have analyzed a mathematical model that describes the motion of an incompressible 
viscous fluid on a compact Riemannian manifold $(\sM, g)$. In case $\sM$ is boundary-less, 
the motion of the fluid is governed by 
\begin{equation}
\label{NS}
\left\{\begin{aligned}
\varrho \left( \partial_t u + \nabla_u  u\right) -\mu_s  (\Delta u +  \Ric u) + \gd\pi &=0 &&\text{on}&&\sM ,\\
\dv u&=0 &&\text{on}&&\sM ,\\
u(0) & = u_0 &&\text{on}&& \sM. \\
\end{aligned}\right.
\end{equation}
Here, the unknowns are the fluid velocity $u$  and the fluid pressure $\pi$.  $\varrho>0$ is the (constant) density, $\mu_s>0$ is the surface shear viscosity.
Moreover, $\nabla_u v$ denotes the Levi-Civita connection (the covariant derivative) for tangent vector fields $u,v \in T\sM$,
$\Delta $ is the connection Laplacian (the negative of the Bochner Laplacian) and $\Ric$ is the Ricci curvature of $\sM$.
In local coordinates, these operators are expressed by
\begin{equation*}
\Delta  u= g^{ij}(\nabla_i \nabla_j -\Gamma_{ij}^k \nabla_k)u , \quad
 \Ric u= R^i_j u^j\frac{\partial} {\partial x^i}, 
\end{equation*}
with $\nabla_j=\nabla_{\frac{\partial}{\partial x^j}}$ being covariant derivatives, and where $\Ric = \Ric^i_j \frac{\partial }{\partial x_i} \otimes dx^j$ is the Ricci $(1,1)$-tensor.
 More details are given in Appendix A.
 
 \medskip\noindent
 In case $\dim \sM=2$, the system \eqref{NS} takes on the form
 \begin{equation}
\label{NS-2D}
\left\{\begin{aligned}
\partial_t u + \nabla_u  u  -\mu_s  (\Delta u + \kappa  u) + \gd\pi &=0 &&\text{on}&&\sM ,\\
\dv u&=0 &&\text{on}&&\sM ,\\
u(0) & = u_0 &&\text{on}&& \sM, \\
\end{aligned}\right.
\end{equation}
where $\kappa$ is the Gaussian curvature of $\sM$.
Here and in the sequel, we set $\varrho=1$.
\goodbreak

\medskip\noindent
It was  shown in \cite{SiWi22} that  any solution of \eqref{NS-2D} with a divergence-free initial value
of finite kinetic energy exists globally and converges exponentially fast  to an equilibrium, that is, to a Killing field.
A similar result was obtained in \cite{SSW25} in case $\sM$ is a compact manifold with boundary, and Navier boundary conditions are imposted \cite{SSW25}.

\medskip
Here we will consider the particular situation where $\sM$ is a sphere in $\R^3$ of radius $a$. In this case, the Killing fields of $\sM$ consist 
exactly of rotations about an axis $(\omega_1, \omega_2,\omega_3)$ and the result above then says that any solution will converge to a steady
motion of rotation around such an axis.

\medskip
As a new ingredient, we will include a Coriolis force on $\sM$ in this paper. Here one might think of $\sM$ being an aqua planet that is completely covered by a fluid 
and is  rotating about the $z$-axis.

\medskip
In geography, the Coriolis force plays an important role in oceanic circulation patterns.
The Coriolis force, also known as the Coriolis effect, is an apparent force experienced by objects moving within a rotating reference frame. This phenomenon arises because of the rotation of the earth and is described mathematically as a deflection of moving objects relative to the planet's surface. Although the Coriolis force is not a true force in the sense of a physical interaction, it is a necessary consideration when analyzing motion in a non-inertial, rotating frame, such as the earth.

The magnitude of the Coriolis force depends on three factors: the speed of the moving object, its latitude, and the angular velocity of the earth's rotation. In the northern hemisphere, the Coriolis force causes moving objects to deflect to the right, while in the southern hemisphere, it causes a deflection to the left. At the equator, the effect is negligible because the rotational velocity is perpendicular to the direction of motion, whereas it becomes strongest at the poles.

\medskip
We will use the following equations to model the effect of the Coriolis force on the motion of an incompressible fluid $u$ 
 \begin{equation}
\label{NS-Coriolis}
\left\{\begin{aligned}
 \partial_t u + \nabla_u  u  -\mu_s  (\Delta u + \kappa  u) + Cu +  \gd\pi &=0 &&\text{on}&&\sM ,\\
\dv u&=0 &&\text{on}&&\sM ,\\
u(0) & = u_0 &&\text{on}&& \sM, \\
\end{aligned}\right.
\end{equation}
where $\sM:= \bS^2_a$ is the sphere in $\R^3$ of radius $a$ centered at the origin.
Here, $C$ represents the effect of the Coriolis force and is given by
\begin{equation*}
Cu = - 2\omega \cos \phi Ku,
\end{equation*}
where $\phi$ denotes the colatitude and $\omega$ is the angular velocity of the rotating reference frame.
The  linear operator $K: T\sM \to T\sM$, acting on tangential vector fields, is defined in (general) local coordinates by
\begin{equation*}
K u = \eps^i_j u^j \frac{\partial}{\partial x_i} = g^{ik}\eps_{kj} u^j  \frac{\partial}{\partial x_i} , \quad\text{for}\quad  u= u^j \frac{\partial}{\partial x_j},
\end{equation*}
where $\eps_{ij}$ are the Levi-Civita symbols, see Appendix~\ref{Appendix-K}.
It will be shown in the Appendix that $K$ rotates vector fields on the sphere $\bS^2_a$
 counterclockwise on the sphere about an axis that is normal to the tangent plane at a chosen point, with outward pointing orientation.

\medskip
The combined effect of the term $(-2\omega \cos \phi) Ku$ shows that tangent vectors $u$ are deflected to the right on the northern hemisphere 
(and to the left on the southern), while this effect vanishes on the equator and is strongest at the poles, due to the factor $\cos \phi$.

\medskip
In the following, we will derive in a rather informal way a characterization of all equilibrium states for  \eqref{NS-Coriolis}.  
A justification for all the steps  and conclusions can be found in the Appendix, see in particular
Appendix~\ref{Appendix-Killing} and Appendix~\ref{Appendix-K}.
We would like to acknowledge that, in addressing the Coriolis force, we benefited from the insights provided in the very nice paper \cite{SaTu20}.

\medskip
Suppose $(u,\pi)$ is a (sufficiently smooth) equilibrium solution of \eqref{NS-Coriolis}. 
Multiplying the first equation in \eqref{NS-Coriolis} with $u$ and integrating over $\sM$ yields
\begin{equation*}
\begin{aligned}
0=\int_\sM (\nabla_u u -\mu_s  (\Delta u + \kappa  u) + Cu +  \gd\pi | u)_g \,d\mu_g 
=  2\mu_s \int |D_u|^2_g \,d\mu_g,
\end{aligned}
\end{equation*}
where we used the metric property of the Levi-Civita connection, $\dv u=0$ and the divergence theorem  to conclude that
$$
\int_\sM (\nabla_ u u | u)_g \,d\mu_g = \frac{1}{2}\int_\sM  \nabla_u (u | u)_g\, d\mu_g =0,
\qquad \int_\sM ( \gd\pi | u)_g \,d\mu_g =0.
$$
In addition, we used that $(Cu |u )_g=0$.
The condition $\| D_u \|_g =0$ implies $D_u=0$, and hence $u$ is a Killing field. 
This in turn yields 
$$
\nabla_u u = -\frac{1}{2} \gd (u | u)_g .
$$
Since  $( \Delta  + \kappa) u = 2\,\dv D_u$, we have  $( \Delta  + \kappa) u =0$ . Therefore, we are left with the equation
\begin{equation*}
-\frac{1}{2} \gd (u | u)_g + Cu + \gd \pi =0.
\end{equation*}
Applying $\rot$ (here we follow an argument given in \cite{SaTu20}) to this equation and using the abbreviation $f= -2 \omega \cos \phi$ and the fact
that $\rot (\gd h)=0$ for any scalar function $h$, we infer that
\begin{equation*}
\rot (Cu) = \dv ( f K^2 u) = -\dv ( f u) = -f  \dv u + (\gd f |u )_g = (\gd f |u )_g  =0.
\end{equation*}
Using spherical coordinates, the  equation $(\gd f |u )_g  =0$ and the fact that $u$ is a Killing field yields 
$$u=u_*:= c \frac{\partial}{\partial \theta},$$ 
where $c$ is a constant and $\theta$ denotes the longitude.
Hence, $u$ is a rotation around the $z$-axis.
With this expression for $u$, we infer that in spherical coordinates
$$
Cu_* = \gd h_*,\quad \text{where}\  h_*= - a^2 c\,\omega \cos^2 \phi.
$$  
This yields the equilibrium pressure 
$$
\pi_* = \frac{1}{2}\left ( c^2 a^2 \sin ^2 \phi + 2 c a^2 \omega \cos^2\phi\right),
$$
where $c$ is a constant.
In summary, the set $\cE_*$ of equilibrium states for \eqref{NS-Coriolis} is given in spherical coordinates by
\begin{equation}
\label{equilibria}
\cE_*= \left\{(u_*,\pi_*): u_* = c \frac{\partial}{\partial\theta},\ \pi_* = \frac{1}{2}\left ( c^2 a^2 \sin ^2 \phi + 2 ca^2 c\,\omega \cos^2 \phi\right)\right\} ,
\end{equation}
where $c$ is a constant.

\medskip\noindent
Here we would like to paraphrase our main result. 
A more precise statement is provided in Theorem~\ref{thm: global}.
\begin{theorem}
\label{thm; global-informal}
For every divergence free initial value $u_0$ in $L_2(\sM; T\sM) $,  equation~\eqref{NS-Coriolis} admits a unique global solution.
$u(t)$ converges at an exponential rate to an equilibrium 
$ u_*= c \frac{\partial}{\partial \theta}$  for some $c$ in $\R$ in the topology of $H^2_2(\sM; T\sM)$ as~$t\to \infty$.

\end{theorem}
Hence, the Coriolis force eventually aligns each solution to a rotation about the $z$-axis.
This result seems rather surprising, as any solution will eventually line up with a rotation about the $z$-axis, no matter in what way it starts out.
In our previous paper \cite{PSW20}, we had shown that in the absence of the Coriolis effect, solutions with initial data in $L_2$ will converge to a rotation about
an axis $ (\omega_1, \omega_2, \omega_3).$ 

\medskip
The motion of fluids on a surface (or a manifold) has attracted attention by many researchers in recent years,
see for instance \cite{CCD17, Czu24, EbMa70, JOR18, KLG17,OQRY18, ReZh13, ReVo18, SSW25, SaTu20, SiWi22}
and the additional references listed in these publications.

\medskip\noindent
We will analyze the system \eqref{NS-Coriolis} by using methods from semi-group theory, maximal regularity, and interpolation-extrapolation spaces.
We would be remiss not to mention the profound influence Giuseppe Da Prato had on the development of this area.
 He was  a pioneer, laying the groundwork for many foundational concepts and advancing the theoretical framework that continues to shape modern 
 research in this area. In particular,  the publication \cite{DaPG75}  laid the groundwork for the  functional analytic approach to maximal regularity.

\medskip\noindent
In Section \ref{sec:Well-posed} we prove that for any initial value $u_0\in L_{2,\sigma}(\sM; T\sM)$ there exists a unique weak solution of \eqref{NS-Coriolis}. By means of the theory of critical spaces, developed e.g. in \cite{PSW18}, it is shown that the solution regularizes instantaneously to a strong $L_p$-$L_q$-solution of \eqref{NS-Coriolis}. 
In Section~\ref{global} we provide a precise statement and a proof of Theorem~\ref{thm; global-informal}.

Finally, in appendices A through E we collect and prove results concerning Riemannian manifolds, Killing vector fields, the rotation operator $K$,
the divergence theorem, the existence of the Helmholtz projection, Korn's inequality, and the $H^\infty$-calculus. 

\section{Mathematical approach}\label{sec:Well-posed}
To analyze \eqref{NS-Coriolis},
we introduce the {\em surface Helmholtz projection}, defined by
$$
\PH u= u-\gd \psi_u, \quad u\in L_q(\sM; T\sM),
$$
where $\gd\psi_u \in L_q(\sM; T\sM)$ is the unique solution of
$$
(\gd \psi_u | \gd \phi)_{\sM}=(u|\gd \phi)_{\sM} ,\quad \forall \phi\in \dot{H}^1_{q'}(\sM),
$$
cf. Proposition~\ref{pro: Helmholtz}.
Here,
$$
 (u | v)_\sM:=  \int_\sM (u|v)_g \, d\mu_g,\quad (u,v)\in L_q(\sM;T\sM)\times L_{q^\prime}(\sM; T\sM),
$$
denotes the  duality pairing between $L_q(\sM; T\sM)$ and $L_{q^\prime}(\sM; T\sM)$, where $\mu_g$ is the volume form induced by $g$.
We note that  in case $q=2$,  the pairing $(\cdot | \cdot )_\sM$ defines an inner product on $L_2(\sM; T\sM)$.

\medskip\noindent
With these preparations, we can introduce the function spaces used in this article
\begin{equation}
\label{spaces}
\begin{aligned}
L_{q,\sigma}(\sM;T\sM):&= \PH L_q(\sM;T\sM) \\
H^s_{q,\sigma}(\sM;T\sM): &= H^s_q(\sM;T\sM) \cap L_{q,\sigma}(\sM;T\sM)  \\
H^{-s}_{q,\sigma}(\sM;T\sM): &=  (H^s_{q',\sigma}(\sM;T\sM))' \\
\end{aligned}
\end{equation}
for   $-1\le s \le 1$ and $1<p,q<\infty$,
where the respective duality parings
\begin{equation*}
\begin{aligned}
&\la \cdot | \cdot \ra_\sM:  H^{-s}_{q,\sigma}(\sM;T\sM)\times H^s_{q',\sigma}(\sM;T\sM)) \to \bR, \\
\end{aligned}
\end{equation*}
are induced by $(\cdot | \cdot )_\sM$.
We would like to point out that our definition of the `negative' spaces $H^{-s}_q$ 
differ from the usual definition in case $-s< -1/{q'}$. This allows for a more streamlined presentation of our results.
Note that
\begin{equation*}
\la u | v \ra_\sM = (u |v)_\sM \quad\text{in case}\quad (u,v)\in L_q(\sM;T\sM)\times L_{q^\prime}(\sM; T\sM).
\end{equation*}
Now we can define the {\em surface Stokes operator}  $A: X_1\to X_0$, by
\begin{equation*}
 Au:= -\mu_s \PH (\Delta  u + \kappa u )
\end{equation*}
with $X_0:= L_{q,\sigma}(\sM;T\sM)$ and
$X_1:=\sD(A):= H^2_{q,\sigma}(\sM;T\sM).$

\medskip\noindent
Equation~\eqref{NS-Coriolis} is equivalent to the equation
\begin{equation}
\label{NS-strong}
 \partial_t u + Au + Bu  = F(u), \quad u(0)=u_0,
\end{equation}
where
\begin{equation}
Bu:=  \PH Cu, \quad F(u): =  \PH \nabla_u u,
\end{equation}
see Remark \ref{rem: q-pressure}.
We have the the following result.
\begin{theorem}
\label{thm: H-strong}
There exists a number $\eta_0>0$ such that 
the operator $\eta + A+B$ admits a bounded $H^\infty$-calculus on $X_0$ with $H^\infty$-angle $<\pi/2$  for all $\eta>\eta_0$. 
\end{theorem}
\begin{proof}
As the operator $B$ is a lower order perturbation of the Stokes operator $A$ the assertion
follows from Theorem 3.1.5 and Corollary 3.3.15 in \cite{PrSi16}.
\end{proof}
\subsection{Weak setting} 
\bigskip
Here we are interested in studying equation \eqref{NS-strong} in a weak setting, so as to be able to admit initial data $u_0\in L_{2,\sigma}(\sM, T\sM)$.
This will be done by casting \eqref{NS-strong}  in an extrapolation setting.
Let 
$$A_0=\eta+A+B, \quad \eta>\eta_0, $$
 and recall that $X_0=L_{q,\sigma}(\sM; T\sM)$. 
By \cite[Theorems V.1.5.1 and V.1.5.4]{Ama95}, the pair $(X_0,A_0)$ generates an interpolation-extrapolation scale $(X_\alpha,A_\alpha)$, $\alpha\in\mathbb{R}$, with respect to the complex interpolation functor. Note that for $\alpha\in (0,1)$, $A_\alpha$ is the $X_\alpha$-realization of $A_0$ (the restriction of $A_0$ to $X_\alpha$) and
$$X_\alpha = D(A_0^\alpha)=[X_0,X_1]_\alpha=H_{q,\sigma}^{2\alpha}(\sM; T\sM), $$
since $A_0$ admits a bounded $H^\infty$-calculus.

Let $X_0^\sharp:=(X_0)'$ and $A_0^\sharp:=(A_0)'$ with $D(A_0^\sharp)=:X_1^\sharp$. Then $(X_0^\sharp,A_0^\sharp)$ generates an interpolation-extrapolation scale $(X_\alpha^\sharp,A_\alpha^\sharp)$, the dual scale, and by \cite[Theorem V.1.5.12]{Ama95}, it holds that
$$(X_\alpha)'=X^\sharp_{-\alpha}\quad\text{and}\quad (A_\alpha)'=A^\sharp_{-\alpha}$$
for $\alpha\in \mathbb{R}$.
Choosing $\alpha=1/2$ in the scale $(X_\alpha,A_\alpha)$, we obtain an operator
$$A_{-1/2}:X_{1/2}\to X_{-1/2},$$
where
$X_{-1/2}=(X_{1/2}^\sharp)'$ (by reflexivity) and, since also $A_0^\sharp$ has a bounded $H^\infty$-calculus,
$$X_{1/2}^\sharp = D((A_0^\sharp)^{1/2})=[X_0^\sharp,X_1^\sharp]_{1/2}
=H_{q',\sigma}^1(\sM; T\sM),$$
with $q'=q/(q-1)$ being the conjugate exponent to $q\in (1,\infty)$.  Moreover, we have $A_{-1/2}=(A_{1/2}^\sharp)'$ and $A_{1/2}^\sharp$ is the restriction of $A_0^\sharp$ to $X_{1/2}^\sharp$. Thus, the operator
$$A_{-1/2}:X_{1/2}\to X_{-1/2}$$
inherits the property of a bounded $H^\infty$-calculus with $H^\infty$-angle $<\pi/2$ from $A_0$,
see Theorem~\ref{thm: H-strong}.

Since $A_{-1/2}$ is the closure of $A_0$ in $X_{-1/2}$ it follows that $A_{-1/2}u=A_0u$ for $u\in X_1=D(A_0)=H_{q,\sigma}^2(\sM; T\sM)$ and thus, for all $v\in X_{1/2}^\sharp$, it holds that
$$
\la A_{-1/2}u,v\ra_\sM=(A_0u|v)_{\sM}=2\mu_s (D_u | D_v)_\sM  +(B u |v)_\sM + \eta (u|v)_\sM,
$$
see Proposition~\ref{pro: divergence thm}.
Using that $X_1$ is dense in $X_{1/2}$, we obtain the identity
$$
\la A_{-1/2}u,v\ra_\sM=2\mu_s (D_u | D_v)_\sM  +(B u |v)_\sM + \eta (u|v)_\sM,
$$
valid for all $(u,v)\in X_{1/2}\times X_{1/2}^\sharp$. 
We call the operator $A^{\sf w}: X_{1/2} \to X_{-1/2}$, given by the representation
$$
\la A^{\sf w}u,v\ra_\sM=2\mu_s( D_u | D_v)_\sM , \quad (u,v)\in X_{1/2}\times X_{1/2}^\sharp,
$$
the \emph{weak Stokes operator} on $\sM .$

Multiplying \eqref{NS-strong} by a function $\phi\in X_{1/2}^\sharp= H_{q',\sigma}^1(\sM; T\sM)$,
we obtain the {\em weak formulation} of \eqref{NS-Coriolis} 
\begin{equation}
\label{NS-weak}
\partial_t u+A ^{\sf w}u + B^{\sf w} u=F^{\sf w}(u),\quad u(0)=u_0,
\end{equation}
in $X_{-1/2}$, where for all for  $(u,\phi)\in H^1_{q,\sigma}(\sM;T\sM) \times H^1_{q',\sigma}(\sM;T\sM)$
$$
\la B^{\sf w} u, \phi \ra_\sM =  (B u | \phi)_\sM,\quad 
\la F^{\sf w}(u),\phi\ra_\sM=(u \otimes u_\flat  | \nabla \phi )_\sM .
$$
Here we used Proposition~\ref{pro: divergence thm}.  
Note that $u$ is a solution of \eqref{NS-weak} if and only if $u$ solves
\begin{equation}
\label{NS-weak-omega}
\partial_t u + \eta u+A ^{\sf w}u + B^{\sf w} u=F^{\sf w}(u)+ \eta u,\quad u(0)=u_0,
\end{equation}
in $X_{-1/2}$.
From the definition of $B$, it follows that
\begin{align*}
\left|  ( B u | \phi)_{\sM} \right|  & \le C  \|u\|_{L_q(\sM)} \|\phi\|_{L_{q'}(\sM)},
\end{align*}
for all $(u,\phi)\in L_{q}(\sM)\times L_{q'}(\sM)$, hence $B^\sw \in \cL(L_{q,\sigma}(\sM;T\sM),  X_{-1/2} )$ is a lower order perturbation of $A^{\sw}$.
We may therefore apply \cite[Corollary~3.3.15]{PrSi16}, to conclude that for some sufficiently large $\eta_0>0$
\begin{equation}
\label{H-infty-weak}
\eta + A^{\sw}+B^{\sw} \in H^\infty( X_{-1/2}) \text{ with $H^\infty$-angle } < \pi/2 \quad \text{for all  }\eta >\eta_0.
\end{equation}
We are now ready to state the main result of this paper concerning existence and uniqueness of solutions for 
 \eqref{NS-weak}.
\begin{theorem}
\label{weak-strong-L2}
For any $u_0\in L_{2,\sigma}(\sM; T\sM)$, problem \eqref{NS-weak} admits a unique solution
\begin{equation*}
u\in H^1_{2}((0, a), H^{-1}_{2,\sigma}(\sM; T\sM))\cap L_{2}((0, a);H^1_{2,\sigma}(\sM; T\sM))
\end{equation*}
for some $a=a(u_0)>0$. The solution exists on a maximal time interval $[0,t^+(u_0))$.
In addition, we have
\begin{equation*}
u\in  C([0,t^+); L_{2,\sigma}(\sM; T\sM))
\end{equation*}
with $t^+=t^+(u_0)$. 
Furthermore, each solution satisfies
\begin{equation*}
u\in H^1_{p,{\rm loc}}((0, t^+); L_{q,\sigma}(\sM; T\sM))\cap L_{p,{\rm loc}}((0, t^+) ; H^2_{q,\sigma}(\sM; T\sM))
\end{equation*}
for any fixed $p,q\in (1,\infty)$.
Therefore, any solution with initial value $u_0\in L_{2,\sigma}(\sM; T\sM)$ regularizes instantaneously and becomes a strong $L_p$-$L_q$ solution {\color{violet}of \eqref{NS-strong}}.
\end{theorem}
\begin{proof}
We will apply Theorem 1.2 in \cite{PSW18} to \eqref{NS-weak-omega} with the choice
 $$X_0^{\sf w}=X_{-1/2}\quad \text{and}\quad X_1^{\sf w}=X_{1/2}.$$
 For that purpose, we will first characterize some relevant interpolation spaces.
In \cite[Section 3.5]{PSW20} we determined the complex interpolation spaces $[X_0^{\sf w}, X_1^{\sf w}]_\theta$ as
\begin{equation}
\label{interpolation}
 \begin{aligned}
 & [X_0^{\sf w}, X_1^{\sf w}]_\theta     \ \, =\  H^{2\theta -1}_{q,\sigma}(\sM; T\sM),\quad && \theta \in (0,1).
\end{aligned}
 \end{equation}
Next we show that the nonlinearity
$F_\sM^{\sf w}:X_\beta^{\sf w} \to X_{0}^{\sf w} $
is well defined, where
$X^{\sf w}_\beta:=[X_0^{\sf w},X_1^{\sf w}]_{\beta} $ for $\beta\in (1/2,1).$
By \eqref{spaces}, \eqref{interpolation} and Sobolev embedding, we have
\begin{equation}
\label{embedding}
X_\beta^{\sf w} \hookrightarrow
H_q^{2\beta-1}(\sM; T\sM) \hookrightarrow L_{2q}(\sM; T\sM),
\end{equation}
provided that $2\beta-1\ge \frac{1}{q}$ (recall that $\dim\sM=2$).
From now on, we assume $2\beta-1=\frac{1}{q}$, which means $q>1$ as $\beta<1$.
Then, by H\"older's inequality and \eqref{embedding}, we obtain
$$|\la F^{\sf w}(u),\phi\ra_\sM|\le \|u\|_{L_{2q}(\sM)}^2 \|\phi\|_{H^1_{q'}(\sM)}
\le C \|u\|^2_{X_\beta^{\sf w}}\,\|\phi\|_{H^1_{q'}(\sM)},$$
showing that
$$F^{\sf w}: X_\beta^{\sf w} \to X^{\sf w}_0\quad\text{with}\quad  \|F_\sM^{\sf w}(u)\|_{X_0^{\sf w}}\le C \|u\|^2_{X^{\sf w}_\beta}.$$
 Thanks to \eqref{H-infty-weak} we can now employ \cite[Theorem 1.2]{PSW18}.
 Indeed, according to \cite{PSW18}, the critical weight $\mu_c$ and the corresponding critical trace space $X_{\gamma,\mu_c}^{\sf w}$ in the weak setting are given by 
$$
X_{\gamma,\mu_c}^{\sf w}=(X_0^{\sf w},X_1^{\sf w})_{\mu_c-1/p,p}, \quad \mu_c:=\mu^{\sf w}_c=1/p+2\beta-1=1/p+1/q,
$$
respectively. In the special case $p=q=2$, this reduces to $\mu_c=1$ and
\begin{equation*} 
X_{\gamma,\mu_c}^{\sf w}= (X_0^{\sf w},X_1^{\sf w})_{1/2,2}=[X_0^{\sf w},X_1^{\sf w}]_{1/2}=L_{2,\sigma}(\sM; T\sM).
\end{equation*}
The validity of the second assertion can be seen as in the proof of \cite[Theorem 4.5]{SiWi22}, taking into account that, by Theorem \ref{thm: H-strong},
 the (shifted) operator $A+B$ in \eqref{NS-strong} has a bounded $H^\infty$-calculus in $X_0=L_{q,\sigma}(\sM;\mathsf{T}\sM)$ with $H^\infty$-angle $<\pi/2$.
\end{proof}
\begin{remark}
\label{rem: q-pressure}
 {\rm (a)} 
Even though the choice $u_0\in L_{2,\sigma}(\sM; T\sM)$ asks for  $(p,q)=(2,2)$, 
proving the regularization of solutions requires the $H^\infty$-calculus  and estimates for the nonlinearity for $q\in (1,\infty)$.

\smallskip\noindent
{\rm (b)} Once the solution $u$ is known, the pressure $\pi$ can be recovered.
Indeed, suppose $u$ solves the equation
$$
\partial_t u + \PH (\nabla_u u -\mu_s \Delta u + Cu)=0, \quad t\in (0, t^+(u_0)).
$$
Then by definition of $\PH$, 
$$\PH (\nabla_u u -\mu_s \Delta u + Cu)=  (\nabla_u u -\mu_s \Delta u + Cu)- \gd \psi_v,$$
where $\psi_v$ solves
$$(\gd \psi_v|\gd \phi)_\sM=(v|\gd \phi)_\sM,\quad \phi\in \dot{H}_{q'}^1(\sM),$$
with $v=  (\nabla_u u -\mu_s \Delta u + Cu)$.
Setting $\pi = -\psi_v$, we see that $(u,\pi)$ is a solution of~\eqref{NS-Coriolis}.
\end{remark}

\section{Global existence and convergence}\label{global}
We remind that 
$$\cE=\left\{c\frac{\partial}{\partial \theta} : c\in\R \right\}$$
is the set of all equilibrium velocities for \eqref{NS-Coriolis}, respectively \eqref{NS-strong}.
We have the following interesting relation.
\begin{lemma}
\label{Null-weak}
$\sN(A^{\sw} +B^{\sw}) =\cE.$
\end{lemma}
\begin{proof}
Pick $u\in \sN(A^{\sw} + B^{\sw})$.
Then
$$0= \la  (  A^{\sw} + B^{\sw})u | u \ra_{\sM} = (D_u | D_u)_\sM + (\PH Cu |u)_M = \| D_u \|^2_{L_2(\sM)},$$
where we employed \eqref{PH-symmetric}, the fact that $\PH u = u$, and \eqref{K-orthogonal} to conclude 
$$ (\PH Cu |u)_g= (Cu |\PH u)=  (Cu | u)_g=0.$$
Hence $D_u=0$, and this shows that $u$ is a Killing field, see \eqref{Killing-equivalent}.
Let $v\in H^1_{2,\sigma}(\sM; T\sM)$ be given. Then we have
\begin{equation*}
0= \la (A^{\sw} + B^{\sw})u |v\ra_\sM = (D_u | D_v)_\sM + (\PH Cu |v)_\sM = (\PH Cu |v)_\sM .
\end{equation*}
Since this is true for any $v\in H^1_{2,\sigma}(\sM ; T\sM)$, by density 
 of $H^1_{2,\sigma}(\sM; T\sM)$ in $L_{2,\sigma}(\sM; T\sM)$ we infer that $\PH Cu=0$, and hence 
$$Cu=\gd h$$ 
for some function $h$.
Applying ${\sf rot}$ to this relation and employing \eqref{K-twice}, \eqref{K-rot} yields 
$$
{\sf rot } (Cu) = \dv (f K^2 u)= - \dv (f u) = - (\gd f | u)_g =0.
$$
We can now infer from Lemma~\ref{u-star} that $u=c\frac{\partial}{\partial_\theta}$ for some $c\in\R$ , 
and hence $u\in\cE$.
This shows that $\sN(A^{\sw} + B^{\sw})\subset \cE$.

\medskip\noindent
Conversely, suppose that  $u\in \cE$. Then it follows from Lemma~\ref{u-star} and the fact   $u$ is a Killing vector field  that 
$
-\mu_s \PH (\Delta u + \kappa  u) + \PH Cu =0,
$
that is, $(A+B)u=0$. Therefore, $(\Aw + B^\sw)u=0$ as well. 
\end{proof}

\noindent
Let  $u_*\in \cE$ be given. Then  we consider the evolution equation

\begin{equation}
\label{NS sys abstract-aux-weak}
\left\{\begin{aligned}
\partial_t v + (A^{\sw} + B^{\sw}) v &=G^{\sw}_*(v),    && t>0 ,\\
v(0) & = v_0, &&
\end{aligned}\right.
\end{equation}
where
$$\langle G^{\sw}_*(v) | \phi\ra_\sM = (  u_* \otimes v_\flat + v\otimes (u_*)_\flat + v\otimes v_\flat | \nabla \phi)_\sM  
$$
for all  $ \phi\in  H^1_{2,\sigma}(\sM; T\sM)$.
Its strong counterpart is given by 
\begin{equation}
\label{NS sys abstract-aux}
\left\{\begin{aligned}
\partial_t v+ (A+B) v &=G_*(v),     && t>0 ,\\
v(0) & = v_0, &&
\end{aligned}\right.
\end{equation}
where $G_*(v):=-\PH ( \nabla_v u_* + \nabla_{u_*} v  + \nabla_v v)$. 

\medskip\noindent
The relation between $G^{\sw}_*(v)$ and $G_*(v)$ is justified by \eqref{div-uv-sigma} and 
Proposition~\ref{pro: divergence thm} (b)(ii).

\goodbreak
\medskip\noindent
Suppose $u$ is a solution of \eqref{NS-weak}. Then  $v= u-u_*$ satisfies
\begin{equation*}
\begin{aligned}
0 &= \la \partial_t v + (A^{\sw} + B^{\sw})v + (A^{\sw} + B^{\sw})u_* - G^{\sw}_*(v) -\nabla_{u_*} u_* |  \phi  \ra_\sM  \\
& = \la \partial_t v + (A^{\sw} + B^{\sw})v - G^{\sw}_*(v) |  \phi  \ra_\sM 
\end{aligned}
\end{equation*}
for  $\phi \in H^{1}_{2,\sigma}(\sM; T\sM)$.
Here we used  Lemma~\ref{Null-weak} to conclude that  $(A^{\sw} + B^{\sw})u_*=0$ and Lemma~\ref{lem: Killing-grad} as well as 
Proposition~\ref{pro: divergence thm} to infer
$$\la \nabla_{u_*} u_* | \phi  \ra_\sM  = (\nabla_{u_*} u_* |  \phi   )_\sM  = \frac{1}{2} (\gd | u_* |^2_g\, | \phi)_\sM =0.$$
This shows that $u$ is a solution of \eqref{NS-weak}  with initial value $u_0$ iff 
 $v= u-u_*$ is a solution of  \eqref{NS sys abstract-aux-weak} with initial value $v_0=u_0-u_*$.
 An analogous statement holds for strong solutions.

\medskip\noindent
Hence, it follows from Theorem~\ref{weak-strong-L2}  that for  every $v_0\in L_{2,\sigma}(\sM;T\sM)$, equation \eqref{NS sys abstract-aux-weak} 
admits a unique solution
\begin{equation}
\label{v-solution-weak-star}
v\in H^1_2((0,t^++); H^{-1}_{2,\sigma}(\sM;T\sM))\cap L_2((0,t^+); H^1_{2,\sigma}(\sM;T\sM))
\end{equation}
for some $t^+=t^+(u_0)>0$.
The solution exists on a maximal time interval $[0,t^+(v_0))$.
In addition, it holds that
\begin{equation}
\label{2D-regularization}
v\in H^1_{p,loc}((0,t^+); L_{q,\sigma}(\sM;T\sM))\cap L_{p,loc}((0,t^+); H^2_{q,\sigma}(\sM;T\sM))
\end{equation}
 for any fixed $p,q\in (1,\infty)$,  and $v$ also solves  \eqref{NS sys abstract-aux}.

\medskip\noindent
With the convention that $H^0_{2,\sigma}(\sM;T\sM):=L_{2,\sigma}(\sM;T\sM)$, we set
\begin{equation}
\label{E-orth}
V^j_2=\{v\in H^j_{2,\sigma}(\sM;T\sM): \, ( v|z)_{\sM}=0 \quad \forall z\in \cE\}, \quad j=0,1.
\end{equation}
As $\cE$ is one-dimensional,  $V^j_2$ is a closed subspace of $H^j_{2,\sigma}(\sM;T\sM)$ and
\begin{equation}
\label{direct sum}
H^j_{2,\sigma}(\sM;T\sM)= \cE \oplus V^j_2, \quad j=0,1,
\end{equation}
by a similar argument to \cite[Remark~4.10(a)]{SiWi22}.

Next we show that any solution of \eqref{NS sys abstract-aux-weak} with an initial value $v_0\in L_{2,\sigma}(\sM)$ that is orthogonal to $\cE$ remains orthogonal
for all later times. Moreover, we establish an energy estimate for such solutions.
\goodbreak
\begin{proposition}
\label{pro: orthogonal-global}
 Given $v_0\in V^0_2$,  let $v$ be the unique solution of \eqref{NS sys abstract-aux-weak}. Then
\begin{itemize}
\item[(a)] $v(t)\in V^1_2$ for all $t\in (0,t^+(v_0))$;
\vspace{1mm}
\item[(b)] there exists a constant $C>0$ such that
\begin{equation}
\label{integral ineq}
\|v(t)\|_{ L_2(\sM) }^2 + C\int_0^t \|v(s)\|_{H^1_2(\sM)}^2\, ds \leq \|v_0\|_{ L_2(\sM) }^2,\quad t\in (0,t^+(v_0)) ;
\end{equation}
\item[(c)] $t^+ (v_0)=+\infty$. Moreover, there exists a constant $\alpha>0$ such that
\begin{equation}
\label{exp growth}
\|v(t)\|_{ L_2(\sM) }\leq e^{-\alpha t} \|v_0\|_{ L_2(\sM) },\quad t\ge 0.
\end{equation}
\end{itemize}
\end{proposition}
\begin{proof}
(a) Pick any  $z\in \cE$.

In the sequel,  we assume that $t\in (0, t^+(v_0))$ is fixed, and we then suppress the time variable and simply write $v$ in lieu of $v(t)$.
We have
\begin{equation}
\label{weak-star-z}
\begin{aligned}
\la \partial_t v | z\ra_\sM = -\la (A^{\sw} + B^{\sw}) v | z\ra_\sM + \la G^{\sw}_*(v) |z\ra_\sM .
\end{aligned}
\end{equation}
As solutions immediately regularize, see \eqref{2D-regularization}, we may replace the weak formulation \eqref{weak-star-z}
by its strong counterpart 
\begin{equation}
\label{strong-star-z}
\begin{aligned}
(\partial_t v | z)_\sM 
&= -((A + B) v | z)_\sM + (G_*(v) |z)_\sM \\
&=  - 2\mu_s (D_v | D_z)_\sM - (C v |z)_{\sM}  +   (G_*(v) |z)_\sM.
\end{aligned}
\end{equation}
Here we employed \eqref{PH-symmetric} and the fact that  $\PH z =z$ to obtain
$$(B v |z)_\sM = (\PH Cv | z)_{\sM} = (C v  | z)_\sM.$$
 For the same reason, we can conclude that
\begin{equation*}
 \langle G_*(v) | z \rangle_{\sM}
=  -(   ( \nabla_{u_*} v | z)_{\sM} + ( \nabla_v u_* | z)_{\sM} + (\nabla_v v | z)_{\sM}.
\end{equation*}
As $z$ is a Killing field, we know that $D_z=0$, see \eqref{Killing-equivalent}, and hence,
$$2\mu_s (D_v | D_z)_\sM =0.$$
Next we show that 
$(Cv |z)_\sM =0. $
Indeed, using \eqref{Kuv}, we have 
$(Cv | z)_g = - (v | Cz)_g.$
As $z\in \cE$, it follows from Lemma~\ref{u-star} that $Cz= \gd h$ for some function $h$.
Hence, 
$$(Cv | z)_\sM = - (v | Cz)_\sM = - (v | \gd h)_\sM =0,$$
where we used \eqref{div-scalar} and the fact that $\dv v=0$ in the last step.

\medskip\noindent
Finally, we claim that $(G_*(v) | z)_\sM=0$. To see this,
 we employ the metric property of the connection to obtain
\begin{align*}
  ( \nabla_{u_*} v | z)_g + ( \nabla_v u_* | z)_g
=    \nabla_{u_*}(  v | z)_g + \nabla_v (  u_* | z)_g  - (  v | \nabla_{u_*} z)_g - (  u_* | \nabla_v z)_g .
\end{align*}
Since $z$ is a Killing vector field, we infer that
\begin{equation*}
(  v | \nabla_{u_*} z)_g + (  u_* | \nabla_v z)_g =0,
\end{equation*}
see \eqref{Killing-def}.
Meanwhile,  Proposition~\ref{pro: divergence thm}  implies
$$
\int_\sM  \left[  \nabla_{u_*}(  v | z)_g + \nabla_v (  u_* | z)_g  \right] \, d\mu_g
 =0.
$$
Using the metric property once more,  we observe that
$$(\nabla_v v | z)_g = \nabla_v (v|z)_g - (v | \nabla_v z)_g. $$
Similar arguments as above yield $(\nabla_v v | z)_\sM =0$.
In summary,  we have shown that
$$
\la \partial_t v (t) | z \ra_{\sM} = \partial_t ( v (t)| z)_{\sM}=0,\quad t\in  (0,t^+(v_0)).
$$
Hence $( v (t)  | z)_{\sM}=0$ and $v(t) \in V^1_2$ for all $t\in [0,t^+(v_0))$.

\medskip\noindent
(b)  Due to  \eqref{2D-regularization}, $v$ is a valid test function in \eqref{NS sys abstract-aux}.
Suppressing the time variable $t\in (0, t^+(v_0))$ we obtain by analogous arguments as in 
\eqref{weak-star-z} and \eqref{strong-star-z}
\begin{equation*}
\la \partial_t v | v\ra_\sM 
=  - 2\mu_s (D_v | D_v)_\sM - (C v |v)_{\sM}  +   (G_*(v) |v)_\sM.
\end{equation*}
If readily follows from \eqref{K-orthogonal} that  $(C v |v)_{\sM}=0$. 
Observing that
\begin{equation*}
\begin{aligned}
- ( G_*(v) | v)_\sM
& =   ( \nabla_{u_*} v | v)_{\sM} + ( \nabla_v u_* | v)_{\sM} + (\nabla_v v | v)_{\sM} \\
& = \frac{1}{2} \nabla_{u_*} |v|^2_g +  (\nabla_v u_* |v)_g + \frac{1}{2} \nabla_v |v|^2_g
\end{aligned}
\end{equation*}
and using the fact that $u_*$ is a Killing  field, $\dv v=\dv u_*=0$ in conjunction with \eqref{div-scalar} implies
$
(G_*(v) | v)_{\sM}=0.
$
We thus have
\begin{equation}
\label{integral ineq 2}
\frac{d}{dt} \|v(t)\|_{ L_2(\sM) }^2 =   -4\mu_s \| D_v\|_{ L_2(\sM) }^2  \le   -C \|v\|_{H^1_2(\sM)}^2 ,
\end{equation}
where the last step follows from  Korn's inequality, cf. Lemma~\ref{Appendix Lem: korn}.
Integrating both sides with respect to time gives \eqref{integral ineq}.

\medskip\noindent
(c) Part (b) shows that
$$
v\in L_2((0,t^+(v_0)); H^1_{2,\sigma} (\sM;T\sM)) .
$$
It follows from \cite[Theorem~2.4]{PSW18} that $t^+(v_0)=+\infty$.
An immediate consequence of \eqref{integral ineq 2}  is
$$
\frac{d}{dt} \|v(t)\|_{ L_2(\sM) }^2 + C \|v\|_{L _2(\sM)}^2  \leq 0 ,\quad \forall t>0.
$$
Solving the above ordinary differential inequality  gives \eqref{exp growth}.
\end{proof}

 We are in a position to state and prove our main theorem.
\begin{theorem}(Global existence) \\
\label{thm: global}
For every $u_0\in L_{2,\sigma}(\sM;T\sM)$, the unique solution $u$ to \eqref{NS-weak}
with initial value $u_0$ 
\begin{itemize}
\item exists globally and enjoys the regularity properties listed in Theorem~\ref{weak-strong-L2}.
\vspace{1mm}
\item For any fixed $q\in (1,\infty)$, $u$ converges to the equilibrium
$$u_* = \cP_{\cE} u_0 = c\frac{\partial}{\partial \theta}$$  
for some $c\in\R$ in the topology of $H^2_{q,\sigma}(\sM;T\sM)$  at an exponential rate as $t\to \infty$,
where $\cP_{\cE}$ denotes the orthogonal projection from $L_{2,\sigma}(\sM;T\sM)$ onto~$\cE$.
\end{itemize}
\end{theorem}
\begin{proof}
In view of  \eqref{direct sum}, we can decompose $u_0$ into $u_0=u_* + v_0$ such that $v_0\in V^0_2$.
Let $v$ be the (unique) solution to
\begin{equation*}
\left\{\begin{aligned}
\partial_t v + (A^{\sw} + B^{\sw}) v &=G^{\sw}_*(v),     && t>0 ,\\
v(0) & = v_0 .&&
\end{aligned}\right.
\end{equation*}
By Proposition~\ref{pro: orthogonal-global}, $v$ exists globally. Then it follows from the considerations preceding the
Proposition that  
$$
u(t) = u_*+v(t),\quad t>0,
$$
is the unique global solution of \eqref{NS-weak} with initial value $u_0$.
As was proved in Proposition~\ref{pro: orthogonal-global}
$$
\|u(t)-u_*\|_{ L_2(\sM) } = \|v(t) \|_{ L_2(\sM) } \leq e^{-\alpha t}\|v_0 \|_{ L_2(\sM) }
=e^{-\alpha t}\|u_0 - u_* \|_{ L_2(\sM) },\quad t>0,
$$
for some $\alpha>0$.
The convergence in the stronger topology $H^2_{q,\sigma}(\sM;T\sM)$ can be proved in the same way as in \cite[Theorem~4.9]{SiWi22}.
\end{proof}

\goodbreak

\appendix

\section{Tensor bundles and the Levi-Civita connection}\label{Appendix A}

In this subsection, we introduce some concepts and results from differential geometry that are used throughout the manuscript. 
Here, we consider the general situation of an $n$-dimensional Riemannian manifold.
 
Let $\sM$ be a compact, smooth $n$-dimensional Riemannian manifold without boundary, $n\ge 2$, and let $(\cdot | \cdot)_g$ denote the Riemann metric on $\sM$.
In the following, we use the same notation as in~\cite{Ama13}.

\medskip
$T{\sM}$ and $T^{\ast}{\sM}$ denote the tangent and the cotangent bundle of ${\sM}$, respectively,
and $T^{\sigma}_{\tau}{\sM}:=T{\sM}^{\otimes{\sigma}}\otimes{T^{\ast}{\sM}^{\otimes{\tau}}}$ stands for the $(\sigma,\tau)$-tensor bundle
of $\sM$ for $\sigma,\tau\in \bN$.
The notations $\Gamma(\sM ;T^{\sigma}_{\tau}{\sM})$ and $\mathcal{T}^{\sigma}_{\tau}{\sM}$ stand  for the set of all sections of   $T^{\sigma}_{\tau}{\sM}$ and the $C^{\infty}({\sM})$-module of all smooth sections of $T^{\sigma}_{\tau}\sM$, respectively.
For abbreviation, we put $\mathbb{J}^{\sigma}:=\{1,2,\ldots,n\}^{\sigma}$, and $\mathbb{J}^{\tau}$ is defined alike.

Given local coordinates $ \{x^1,\ldots,x^n\}$,
$$(i):=(i_1,\cdots,i_{\sigma})\in\bJ^{\sigma},\quad (j):=(j_1,\cdots,j_{\tau})\in\bJ^{\tau},$$
we set
\begin{align*}
\frac{\partial}{\partial{x}^{(i)}}:=\frac{\partial}{\partial{x^{i_1}}}\otimes\cdots\otimes\frac{\partial}{\partial{x^{i_{\sigma}}}}, \hspace*{.5em} dx^{(j)}:=dx^{j_1}\otimes{\cdots}\otimes{dx}^{j_{\tau}}.
\end{align*}
Suppose that $a\in \Gamma(\sM; T^\sigma_\tau \sM) $.
The local representation of $a$ with respect to these coordinates is given by
\begin{align*}
a=a^{(i)}_{(j)} \frac{\partial}{\partial{x}^{(i)}} \otimes dx^{(j)}, \hspace{1em}\text{ with } a^{(i)}_{(j)}: U_k \to \bK,
\end{align*}
where  $U_k\subset  \sM$  is a coordinate patch.

For $s\in \{1,\ldots, \sigma \}$,  $t\in \{1,\ldots, \tau\}$ and $a\in \Gamma(\sM; T^\sigma_\tau \sM)$,
${\sf C}^s_t (a)\in \Gamma(\sM; T^{\sigma-1}_{\tau-1} \sM)$ denotes the contraction of  $a$ with respect to the $(s,t)$-position.
This means that  in a local representation of $a$,
$$
a=a^{(i_1,\dots, i_s,\dots, i_\sigma)}_{(j_1,\dots, j_t,\dots, j_\tau)}
\frac{\partial}{\partial x^{i_1}}\otimes\cdots \otimes \frac{\partial}{\partial x^{i_s}}\otimes \cdots \otimes \frac{\partial}{\partial x^{i_\sigma}}
\otimes dx^{j_1}\otimes \cdots \otimes dx^{j_t}\otimes \cdots \otimes dx^{j_\tau},
$$
the terms $\frac{\partial}{\partial x^{i_s}}$ and $dx^{j_t}$ are deleted and
$
a^{(i_1,\dots, i_s,\dots, i_\sigma)}_{(j_1,\dots, j_t,\dots, j_\tau)}$ is replaced by $ a^{(i_1,\dots, k,\dots, i_\sigma)}_{(j_1,\dots, k,\dots, j_\tau)}$,
and the sum convention is used for $k$.

\medskip
For $a\in\Gamma(\sM; T^\sigma_{\tau}\sM)$, $\tau\ge 1$, $a^\sharp \in\Gamma(\sM; T^{\sigma+1}_{\tau-1}\sM)$ is defined by
\begin{equation*}
a^\sharp := g^\sharp a:={\sf C}^{\sigma+2}_1(a \otimes   g^{*} ),
\end{equation*}
and for $a\in\Gamma(\sM; T^\sigma_{\tau}\sM)$, $\sigma\ge 1$,  $a_\flat \in\Gamma(\sM; T^{\sigma-1}_{\tau+1}\sM)$
is defined by
\begin{equation*}
a_\flat := g_\flat a:={\sf C}^{\sigma }_1( g\otimes  a  ).
\end{equation*}


Any
$S\in  \Gamma(\sM;T^1_1\sM)$ induces a linear map from $\Gamma(\sM;T\sM)$
to $\Gamma(\sM;T\sM)$ by virtue of
$$
S u=(S^i_j \frac{\partial}{\partial x^i}\otimes dx^j )u = S^i_j u^j  \frac{\partial}{\partial x^i},
\quad u=u^j \frac{\partial}{\partial x^j}\in\Gamma(\sM;T\sM).
$$
The dual $S^*$ of $S\in  \Gamma(\sM;T^1_1\sM)$  is a linear map from  $\Gamma(\sM;T^*\sM)$ to $\Gamma(\sM;T^*\sM)$,
defined by
$$
S^* \alpha=(S^i_j dx^j \otimes \frac{\partial}{\partial x^i})\alpha = S^i_j  \alpha_i dx^j,
 \quad \alpha  = \alpha_i dx^i\in \Gamma(\sM;T^*\sM).
$$
The adjoint $S^\sT$ of $S\in  \Gamma(\sM;T^1_1\sM)$
is the linear map from $\Gamma(\sM;T\sM)$ to $\Gamma(\sM;T\sM)$ defined by $S^{\sT} =\gs S^* \gf$, or more precisely,
\begin{equation*}
S^\sT u = \gs[ S^* (\gf u)], \quad u \in \Gamma (\sM; T\sM).
\end{equation*}
It holds that $(Su|v)_g=(u|S^{T}v)_g$ for  tangent fields $u,v$.
In local coordinates,
$
S^\sT = g^{i\ell} S^m_\ell g_{jm} \frac{\partial}{\partial x^i} \otimes dx^j.
$

Let $\nabla$ be the Levi-Civita connection on $\sM$.
For $u\in C^1(\sM; T\sM)$, the covariant derivative $\nabla u\in C(\sM; T^1_1\sM)$ is given in local coordinates by
$$
\nabla u= \nabla_j u\otimes dx^j= (\partial_j u^i +\Gamma^i_{jk}u^k)\frac{\partial}{\partial x^i}\otimes dx^j =:u^i_{|j}\frac{\partial}{\partial x^i}\otimes dx^j,
$$
where $u=u^i\frac{\partial}{\partial x^i}$,
$\nabla_j = \nabla_{\frac{\partial}{\partial x^j}}$,
 and $\Gamma^i_{jk}$ are the Christoffel symbols.
It follows that  $\nabla u + [\nabla u]^{T} $ is given in local coordinates by
\begin{equation}
\label{AA}
D_u:=\nabla u + [\nabla u]^{\sT}= \big( u^i_{|j} + g^{i\ell} u^m_{|\ell} g_{jm}\big) \frac{\partial}{\partial x^i}\otimes dx^j
\end{equation}
and
\begin{equation}
\label{AB}
D(u):=(\nabla u + [\nabla u]^{\sT})^\sharp =\left(g^{jk} u^i_{|k} + g^{ik} u^j_{|k}\right)\frac{\partial}{\partial x^i}\otimes \frac{\partial }{\partial x^j}.
\end{equation}

The extension of the Levi-Civita connection on $C^1(\sM; T^{\sigma}_{\tau}{\sM})$  is again denoted by $\nabla:=\nabla_g$.
For $a\in C^1(\sM; T^{\sigma}_\tau \sM)$, $\nabla a \in C(\sM; T^{\sigma}_{\tau+1}\sM)$ is given in local coordinates by
$\nabla a = \nabla_j a \otimes dx^j$,  and
$$
\dv : C^1(\sM; T^{\sigma }_\tau \sM) \to C(\sM; T^{\sigma -1}_\tau \sM), \quad \sigma\ge 1,
$$
is the divergence operator, defined by
$
\dv a = {\sf C}^{\sigma}_{\tau +1}(\nabla a).
$
In particular,
$$
\dv u= u^i_{|i} \quad \text{for}\quad u=u^i \frac{\partial}{\partial x^i}
$$
and
\begin{equation*}
\begin{aligned}
\dv S  = S^{ik}_{|k} \frac{\partial} {\partial x^i} =\left( \partial_k S^{ik} + \Gamma^i_{kl} S^{lk} + \Gamma^k_{kl}S^{il}\right)\frac{\partial} {\partial x^i}\quad \text{for}\quad S=S^{ij} \frac{\partial} {\partial x^i} \otimes  \frac{\partial} {\partial x^j}.
\end{aligned}
\end{equation*}
This implies 
\begin{equation*}
\label{div-uv}
\dv (u\otimes v) = \nabla_v u + (\dv v)u, 
\end{equation*}
 where
$u\otimes v= u^iv^j \frac{\partial} {\partial x^i} \otimes  \frac{\partial} {\partial x^j}$.
 In particular,
\begin{equation}
\label{div-uv-sigma}
 \dv (u\otimes v) = \nabla_v u \quad \text{in case}\quad  \dv v =0.
\end{equation}
The following relation is well-known, see for instance~\cite[Lemma 2.1]{SaTu20},
\begin{equation*}
2\,\dv D(u) = (\Delta + \Ric) u.
\end{equation*}
In case $\dim \sM=2$, one has
\begin{equation*}
2\,\dv D(u)= (\Delta + \kappa) u,
\end{equation*}
where $\kappa$ is the Gaussian curvature of $\sM$.
For a scalar function $h \in C^1(\sM; \bK)$, the gradient vector $\gd h\in C(\sM; T\sM)$ is defined by the relation
$$
(\gd h | u)_g := \la \nabla h , u\ra_g= \nabla_u h , \quad u\in C(\sM; T\sM),
$$
where $\nabla h\in C(\sM; T^\ast \sM)$ is the covariant derivative of $h$
and $\la \cdot, \cdot\ra_g$ denotes the duality pairing between $T\sM$ and $T^*\sM$.  In local coordinates, we have
$$
(\gd h)^i = g^{ij} \partial_j h ,\quad 1\le i \le n.
$$

%
%
%

\section{Killing fields}\label{Appendix-Killing}
 A vector field $u\in C^1(\sM; T\sM)$ is called a \emph{Killing field} if
\begin{equation}
\label{Killing-def}
(\nabla_v u | w)_g + (\nabla_w u | v)_g =0,\quad v,w\in \Gamma(\sM ; T\sM).
\end{equation}
It is not difficult to see that 
\begin{equation}
\label{Killing-equivalent}
\text{$u$ is a Killing vector field} \iff D_u=0,
\end{equation}
where $D_u = \nabla u + (\nabla u)^{\sf T}$.
If follows from  \eqref{AB} that $u$ is a Killing field iff
\begin{equation}
\label{Killing-local}
g^{jk} u^i_{|k} + g^{ik} u^j_{|k} =0.
\end{equation}
\begin{lemma}
\label{lem: Killing-grad}
Suppose $u,v$ are Killing fields. Then
$$\nabla _u v+ \nabla_v u =  -\gd (u|v)_g . $$
In particular, for any Killing field $u$, 
$$ \nabla_u u = -\frac{1}{2} \gd \| u \|^2_g. $$
\end{lemma}
\begin{proof}
It follows from \eqref{Killing-equivalent} that
\begin{equation*}
 \nabla_u v + \nabla_{v} u =  (\nabla v)u + (\nabla u ) v  
= -  ((\nabla v)^\sT u + (\nabla u )^\sT v).    
\end{equation*}
According to \eqref{AA},
\begin{equation*}
 (\nabla v)^\sT u + (\nabla u )^\sT v = 
 g^{i\ell} (u^m_{| \ell } v^j  + v^m_{| \ell } u^j) g_{jm} \frac{\partial}{\partial x^i}.
\end{equation*}
On the other hand, using the metric property of the connection, 
\begin{equation*}
\begin{aligned}
\gd (u | v)_g &= g^{i\ell} \nabla_\ell (u | v)_g \frac{\partial}{\partial x^i}
= g^{i\ell}\left( (\nabla_\ell u | v)_g  +  (u | \nabla_\ell v)_g\right)\frac{\partial}{\partial x^i} \\
&= g^{i\ell} (u^m_{| \ell } v^j  + v^m_{| \ell } u^j) g_{jm} \frac{\partial}{\partial x^i}.
\end{aligned}
\end{equation*}
Hence the assertion follows.
\end{proof}
\section{The rotation operator K}\label{Appendix-K}
Here we explicitly assume that $\dim \sM =2$. 
Then the \emph{Levi-Civita symbols} $\eps_{ij}$ are defined by
\begin{equation*}
\epsilon_{ij} = \sqrt{\det(g)} \begin{cases}
1, & \text{if } (i, j) = (1, 2), \\
-1, & \text{if } (i, j) = (2, 1), \\
0, & \text{if } i = j.
\end{cases}
\end{equation*}
Suppose $u\in \Gamma (\sM;  T\sM)$. Then the \emph{rotation operator} $K$ is defined by
\begin{equation*}
Ku = \epsilon^k_j u^j \frac{\partial}{\partial x^k} = g^{ki} \epsilon_{ij} u^j \frac{\partial}{\partial x^k} .
\end{equation*}
One readily verifies that
\begin{equation}
\label{K-orthogonal}
(Ku | u)_g =0,\quad u\in \Gamma(\sM; T\sM).
\end{equation}
Hence, $u$ and $Ku$ are orthogonal. 
One might think of $K$ as rotating vectors by an angle of $90$ degrees.
A straightforward computation also shows that
\begin{equation}
\label{Kuv}
(Ku | v)_g= - (u | Kv)_g,\quad  u,v\in \Gamma(\sM; T\sM) , 
\end{equation}
and
\begin{equation}
\label{K-twice}
K^2u= -u,\quad u\in \Gamma(\sM; T\sM) .
\end{equation}
The  \emph{vorticity of a vector field} is defined by
\begin{equation}
\label{K-rot}
{\sf rot\,} u = \dv (Ku),\quad u\in C^1(\sM; T\sM). 
\end{equation}
One shows that
\begin{equation*}
{\sf \rot}(\gd h)=0,\quad h\in C^2(\sM).
\end{equation*}
Suppose $\sM = \bS^2_a$, the round sphere in $\R^3$ of radius $a$, centered at the origin.
Let 
\begin{equation*}
x= a\sin \phi \cos \theta,\quad y=a \sin \phi \sin \theta, \quad z= a\cos \phi, \quad \theta \in [0, 2\pi),\ \phi\in (0,\pi),
\end{equation*}
be spherical coordinates for $\bS^2_a$. 
Here, $\theta$ corresponds to the latitude and $\phi$ to the co-latitude, respectively, 
where $\phi=0$ at the north pole.
We then have with the assignment $\theta \leftrightarrow 1$ and $\phi \leftrightarrow 2$
\begin{equation*}
[g]=[g_{ij}] = 
\begin{bmatrix}
& a^2 \sin^2 \phi &0 \\
&0                         &a^2
\end{bmatrix}.
\end{equation*}
One then readily verifies that
\begin{equation*}
Ku = a^2 \sin \phi \left(\frac{1}{a^2\sin^2 \phi} u^\phi \frac{\partial}{\partial \theta} 
- \frac{1}{a^2} u^\theta  \frac{\partial}{\partial \phi}
\right), \quad u = u^\theta  \frac{\partial}{\partial \theta} + u^\phi  \frac{\partial}{\partial \phi}.
\end{equation*}
In particular,
\begin{equation*}
K \left(\frac{\partial}{\partial \theta}\right) = -  \sin \phi  \frac{\partial}{\partial \phi},
\quad
K \left(\frac{\partial}{\partial \phi }\right)
=     \frac{1}{ \sin \phi}  \frac{\partial}{\partial \theta}. 
\end{equation*}
Hence, $K$ rotates vectors  in a counterclockwise direction: vectors directed eastward are turned northward, and vectors directed southward are turned eastward.

\medskip\noindent
In spherical coordinates, the Coriolis term $C$ is given by
\begin{equation}
\label{Cu-spherical}
Cu=-2\omega \cos\phi Ku 
= 2  a^2 \omega \sin \phi \cos \phi \left(\frac{-1}{a^2\sin^2 \phi} u^\phi \frac{\partial}{\partial \theta} 
+  \frac{1}{a^2} u^\theta  \frac{\partial}{\partial \phi}\right)
\end{equation}
for  $u = u^\theta  \frac{\partial}{\partial \theta} + u^\phi  \frac{\partial}{\partial \phi}$.
One can now infer that $C$ rotates vectors in a clockwise direction on the northern hemisphere; that is, vectors are deflected to the right
on the northern hemisphere
(and to the left on the southern hemisphere). Moreover, the effect of $C$ vanishes at the equator and is strongest 
at the poles. This precisely captures the effect of the Coriolis force.
\begin{lemma} 
\label{u-star}
Suppose $u$ is a Killing field and $(\gd f  | u)_g=0$, where $f= -2 \omega \cos \phi$.
Then 
$
u=u_* = c\frac{\partial }{\partial \theta},
$
where $c$ is a constant, and 
$$
Cu_*= \gd h_*,\quad\text{where}\quad h_*= - a^2 c\omega \cos^2 \phi.
$$ 
\end{lemma}
\begin{proof}
We obtain in spherical coordinates for $u=u^\theta \frac{\partial}{\partial \theta} + u^\phi \frac{\partial}{\partial \phi}$
$$
(\gd \cos \phi | u)_g =  (\partial_\theta  \cos \phi ) u^\theta + (\partial_\phi \cos \phi) u^\phi =0.
$$
This implies $u= u^\theta \frac{\partial}{\partial \theta}$. Since $u$ is a Killing field, we conclude that 
$u=u_*= c  \frac{\partial}{\partial \theta}$, where $c$ is a constant.

This can, for instance, be derived from  \eqref{Killing-local}.
Indeed, in spherical coordinates we have for the Christoffel symbols
\begin{align*}
[\Gamma^1_{ij}]= 
\begin{bmatrix}
& 0 & \cot \phi \\
&\cot \phi  & 0 
\end{bmatrix},
\quad
[\Gamma^{2}_{ij}] =
\begin{bmatrix}
& -\sin \phi \cos \phi &0 \\ 
&0   &0
\end{bmatrix} ,
\end{align*}
where we used the convention $\theta \leftrightarrow 1$ and $\phi \leftrightarrow 2$.
We will now employ \eqref{Killing-local},
$$
g^{jk} u^i_{|k} + g^{ik} u^j_{|k} = g^{jk}(\partial_k u^i +  \Gamma^i_{kl}  u^l) + g^{ik}(\partial_k u^j + \Gamma^j_{kl} u^l)=0,
$$
for the vector field 
$u=u^1\frac{\partial}{\partial x^1}= u^\theta \frac{\partial}{\partial \theta}.$
Choosing $i=j=1$, we obtain
$$
0= g^{jk} u^i_{|k} + g^{ik} u^j_{|k} =
2 g^{11}(\partial_1 u^1 + \Gamma^1_{11}u^1)=  \frac{2}{a^2 \sin^2 \phi} \partial_\theta u^\theta ,
$$
as $\Gamma^1_{11}=0$. Hence  $\partial_\theta u^\theta=0$.
Choosing $i=1$ and $j=2$, a simple computation yields
\begin{equation*}
\begin{aligned}
0= g^{jk} u^i_{|k} + g^{ik} u^j_{|k} 
& =g^{22}(\partial_2 u^1 +\Gamma^1_{21} u^1) + g^{11}(\partial_1 u^2 + \Gamma^2_{11}u^1)  \\
& = g^{22}\partial_2 u^1 + (g^{22}\Gamma^1_{21} + g^{11} \Gamma^2_{11}) u^1 \\
 &= \frac{1}{a^2} \partial_\phi u^\theta.
\end{aligned}
\end{equation*}
Hence, $\partial_\phi u^\theta=0$ as well. This implies $u^\theta =c$ for some constant $c$.

\medskip\noindent
Employing \eqref{Cu-spherical} results in 
\[
Cu_* =  2 c\omega   \sin \phi \cos \phi  \frac{\partial}{\partial \phi} .
\]
Choosing $h_* = - a^2  c\omega \cos^2 \phi$ then yields the assertion.
\end{proof}


\section{divergence theorem, Helmholtz projection, \\ Korn's inequality}\label{Appendix D}

\begin{proposition} [Divergence theorem]
\label{pro: divergence thm}  
\phantom{new line}

\smallskip\noindent
Suppose $u,v,w\in C(\sM; T\sM)$ and $\phi\in C(\sM)$ are sufficiently regular.  Then

\begin{enumerate}
\item[]
\vspace{-4mm}
\item[{\rm (a)}]
\begin{equation}
\label{div-scalar}
\begin{aligned}
 \int_{\sM} (\dv u)h  \,d\mu_g = -\int_{\sM} (u | \gd h )_g \, d\mu_g  = - \int_{\sM} \nabla_u h \,d\mu_g ,
\end{aligned}
\end{equation}
where  $\mu_g$ is the volume element induced by $g$.
\vspace{2mm}
\item[{\rm (b)}] 
\begin{enumerate}
\vspace{2mm}
\item[{\rm (i)}] $((\Delta  + \Ric)u | v)_\sM = - 2(D_u | D_v)_\sM$.
where $D_u= \frac{1}{2} (\nabla u + [\nabla u]^{\sf T})$ and $D_v$ is defined analogously. 
\vspace{2mm}
\item[{\rm (ii)}]  $(\dv (u\otimes v) | w)_{\sM} =  - (u\otimes v_\flat | \nabla w)_{\sM}.$ 
\end{enumerate}
\end{enumerate}
\end{proposition}
\begin{proof}
For a proof, we refer to~\cite[Lemma B.1]{SSW25} (where a more general situation involving manifolds with boundaries is considered).
\end{proof}

\begin{proposition}
\label{pro: Helmholtz}
Let $1<q<\infty$. For each $u\in L_q(\sM; T\sM)$ there exists a unique
solution $\gd \psi_u \in L_q(\sM; T\sM)$ of
$$(\gd \psi_u|\gd \phi)_\sM=(u|\gd \phi)_\sM,\quad \phi\in \dot{H}_{q'}^1(\sM).$$
The solution satisfies
$$\|\gd \psi_u \|_{Lq(\sM)}\le C \| u\|_{Lq(\sM)}. $$
\end{proposition}
\begin{proof}
For a proof, we refer to \cite[Lemma B.6]{SSW25} (where a more general situation is considered).
\end{proof}
If follows  that the
\emph{Helmholtz projection} 
\begin{equation*}
\PH u := u -\gd \psi_u: L_q(\sM; T\sM)\to L_q(\sM; T\sM)
\end{equation*}
is well defined and  continuous.

\medskip\noindent
For any $u\in L_q(\sM; T\sM)$ and $v\in L_{q'}(\sM; T\sM)$ it holds
\begin{equation}
\label{PH-symmetric}
\begin{aligned}
(\PH u | v)_{\sM}&= (u - \gd \psi_u | v)_{\sM} = (u|v)_{\sM} - (\gd \psi_u | v)_{\sM}\\
&= (u|v)_{\sM} - (\gd \psi_u | \gd \psi_v)_{\sM}= (u|v)_{\sM} - (u | \gd \psi_v)_{\sM}\\
&= (u| \PH v)_{\sM}
\end{aligned}
\end{equation}
as $\psi_u\in \dot{H}^1_q(\sM)$ and $\psi_v \in \dot{H}^1_{q'}(\sM)$.

\begin{lemma}[Korn's inequality]
\label{Appendix Lem: korn}
There exists some constant $C>0$ such that
\begin{equation*}
\|v\|_{H^1_2(\sM)} \leq C \| D_v\|_{L_2(\sM)} ,\quad v\in  V^1_2,
\end{equation*}
 where $V^1_2$ is defined in \eqref{E-orth}.
 \end{lemma}
 \begin{proof}
 For a proof, we refer again to \cite[Lemma B.2]{SSW25} (where a more general situation is considered).
 \end{proof}
\section{$H^\infty$-calculus}
\label{Appendix E}
For the reader's convenience, we include  here the definition of \emph{bounded $H^\infty$-calculus}, and  we refer to 
\cite{PrSi16} for more details.
For $\theta\in (0,\pi]$, the open sector with angle $2\theta$ is denoted by
$$\Sigma_\theta:= \{\omega\in \mathbb{C}\setminus \{0\}: |\arg \omega|<\theta \}. $$
\begin{defi}
\label{Def sectorial}
Let $X$ be a complex Banach space, and $\cA$ be a densely defined closed linear operator in $X$ with dense range. $\cA$ is called sectorial if $\Sigma_\theta \subset \rho(-\cA)$ for some $\theta>0$ and
$$ \sup\{\|\mu(\mu+\cA)^{-1}\|_{\cL(X)} : \mu\in \Sigma_\theta \}<\infty. $$
The class of sectorial operators in $X$ is denoted by ${\mathcal S}(X)$.
The spectral angle $\phi_\cA$ of $\cA$ is defined by
$$
\phi_\cA:=\inf\{\phi:\, \Sigma_{\pi-\phi}\subset \rho(-\cA),\, \sup\limits_{\mu\in \Sigma_{\pi-\phi}} \|\mu(\mu+\cA)^{-1} \|_{\cL(X)}<\infty \}.
$$
\end{defi}
Let $\phi\in (0,\pi]$. Then
$$H^\infty(\Sigma_\phi):=\left\{f: \Sigma_\phi \to \mathbb{C}: f \text{ is analytic and } \|f\|_\infty<\infty \right\} $$
and
$$\cH_0(\Sigma_\phi) = \left\{f\in H^\infty(\Sigma_\phi): \exists s>0, c>0 \text{ s.t. } |f(z)| \leq c\frac{|z|^s}{1 +|z|^{2s}} \right\}. $$

\begin{defi}
\label{Def: cHi}
Suppose that $\cA\in \cS(X)$. Then $\cA$ is said to admit a bounded $H^\infty$-calculus if there are $\phi>\phi_\cA$ and a constant $K_\phi$ such that
\begin{equation}
\label{Appendix B: cHi}
\|f(\cA) \|_{\cL(X)} \leq K_\phi \|f\|_{\infty} ,\quad f\in \cH_0(\Sigma_{\pi-\phi}).
\end{equation}
Here
\begin{equation}
\label{Def integral contour}
f(\cA):=-\frac{1}{2\pi i}\int_\Gamma (\lambda + \cA)^{-1} f(\lambda) \, d\lambda, \quad
\Gamma=
\begin{cases}
-t e^{-i \theta} \quad & \text{for } t<0, \\
t e^{ i \theta}   & \text{for } t\geq 0,
\end{cases}
\end{equation}
is a positively oriented contour
for any $\theta\in (0,\pi-\phi  ) $. 
The $H^\infty$-angle of $\cA$ is defined as
$$\phi^\infty_\cA:=\inf\{\phi>\phi_\cA: \eqref{Appendix B: cHi} \text{ holds}\}.$$
\end{defi}


\bigskip\noindent
{\bf Data Availability Statement:} Data sharing is not applicable, as no new datasets were generated or analyzed for this article.

\medskip\noindent
{\bf Conflict of interest:}  The authors assert that there is no conflict of interest to declare.

\goodbreak

\end{document}